\documentclass[12pt]{article}

\usepackage{amsmath, amssymb}
\usepackage{amsthm}
\usepackage{mathrsfs}
\usepackage[margin=1in]{geometry}
\usepackage{ifpdf}
\ifpdf
  \usepackage[hidelinks]{hyperref}
\else
  \usepackage[dvipdfmx,hidelinks]{hyperref}
\fi

\title{Note on Morita equivalence in ring extensions}
\author{Satoshi Yamanaka\\[0.25em]
\small Department of Integrated Science and Technology\\
\small National Institute of Technology, Tsuyama College\\
\small Okayama 708-8509, JAPAN\\
\small \texttt{yamanaka@tsuyama.kosen-ac.jp}}
\date{}

\newtheorem{Theorem}{\quad Theorem}[section]

\newtheorem{Proposition}[Theorem]{\quad Proposition}

\newtheorem{Lemma}[Theorem]{\quad Lemma}

\begin{document}

\maketitle

\begin{abstract}
It seems that Morita invariance is a useful criterion for judging the importance of the classes of ring extensions concerned. 
Y. Miyashita introduced the notion of  Morita equivalence in ring extensions, and he showed that the classes of $G$-Galois extensions and  Frobenius extensions are Morita invariant. After that, S. Ikehata  showed that the classes of separable extensions, Hirata separable extensions, symmetric extensions, and QF-extensions are Morita invariant. In this paper, we shall prove that the classes of several extensions are Morita invariant. Further, we will give an example of the class of ring extensions which is not Morita invariant. 
\end{abstract}

\vspace{0.5cm}
\noindent
{\small {\bf Note for the arXiv version.} This manuscript is the author's accepted manuscript of the article ``Note on Morita Equivalence in Ring Extensions,'' published in {\it Communications in Algebra} {\bf 44} (2016), no.~9, 4121--4131. The content of this manuscript corresponds to the published article. The version of record is available at \href{https://doi.org/10.1080/00927872.2015.1044098}{10.1080/00927872.2015.1044098}.}
\vspace{0.5cm}



\section{Introduction and Preliminaries}

Throughout this paper, 
every ring  has identity 1, its subring  contains 1, and every module over a ring is unital. 
Let $A$ and $A'$ be rings, and let ${}_AM_{A'}$ and  ${}_AN_{A'}$ be $A$-$A'$-modules. 
We write ${}_AM_{A'} | {}_AN_{A'}$ if ${}_AM_{A'}$ is isomorphic to a direct summand of a finite direct sum of copies of ${}_AN_{A'}$. 
If ${}_AM_{A'} | {}_AN_{A'}$ and ${}_AN_{A'}|{}_AM_{A'}$, then we write ${}_AM_{A'} \sim {}_AN_{A'}$. 
By ${\rm Hom}^{r}({}_AM,{}_AN)$  we denote the module of all left $A$-homomorphisms 
 from $M$ to $N$ acting on the right side of $M$. 
Similarly, by ${\rm Hom}^{\ell}(M_{A'},N_{A'})$ we denote the module of all right $A'$-homomorphisms 
from $M$ to $N$ acting on the left side of $M$. 
${}_AM_{A'}$ is called a {\it Morita module} if ${}_AM \sim {}_AA$ and ${\rm End}^r({}_AM)=A'$. 

In this paper, $A/B$ will represent a ring extension, that is, $B$ is a subring of $A$.
Let $A/B$ be a  ring extension. 
For a subset $X$ of $A$ and an $A$-$A$-module $S$,  we denote $V_A(X) = \{ a \in A \, | \, x a = a x \  {\rm for \ all }  \ x \in X \}$
and $S^A = \{ s \in S \, | \, \alpha s = s \alpha \ {\rm for \ all} \ \alpha \in A \}$. 
In particular, $V_A(B)$ is the centralizer of $B$ in $A$ and $V_A(A)$ is the center of $A$. 
$A/B$ is called a {\it trivial extension}  if  there is a $B$-$B$-module $S$ such that $A=B \oplus S$
 and the multiplication in $A$ is given by $(b,s)(c,t)=(bc,bt+sc)$.   
The trivial extension is one of the classical ring extensions. 
$A/B$ is called a {\it liberal extension}  if there is a  finite set of elements $\{v_1, \, v_2,  \cdots, \, v_n\}$ of $V_A(B)$
 such that $A=\sum_{i=1}^n v_iB$ in the sense of J. C. Robson and L. W. Small (cf. \cite{RS}).  
$A/B$ is called a {\it left} (resp. {\it right}) {\it depth two extension} if ${}_BA\otimes_BA_A | {}_BA_A$ 
(resp. ${}_AA\otimes_BA_B | {}_AA_B$) (cf. \cite{KS}). 
L. Kadison and K. Szlach{\'a}nyi characterized left and right depth two extensions as follows : 

\begin{Proposition}\label{DT}{\rm \cite[Lemma 3.7]{KS}}
A ring extension $A/B$ is left $($resp. right$)$ depth two if and only if there exist 
$t_i \in (A\otimes_BA)^B$ and $\beta_i \in {\rm End}^\ell({}_BA_B)$ 
such that $\sum_i t_i \beta_i(x)y=x \otimes y$ (resp. $\sum_i x\beta_i(y) t_i =x \otimes y$) for all $x$, $y \in A$.
\end{Proposition}

For a ring extension $A/B$, we set  $V=V_A(B)$ and $C=V_A(A)$. 
$A/B$ is called a {\it separable extension} if the $A$-$A$-homomorphism of $A \otimes_BA$ onto $A$
defined by $a \otimes b \to ab$ splits, and 
$A/B$ is called a {\it Hirata separable extension} if $A \otimes_BA$ is
$A$-$A$-isomorphic to a direct summand of a finite direct sum of copies of $A$, 
or equivalently, 
$V$ is finitely generated projective as $C$-module and ${}_AA\otimes_BA_A \cong {}_A{\rm Hom}^{\ell}(V_C, A_C)_A$. 
The notion of Hirata separable extensions was introduced by K. Hirata as a generalization of Azumaya algebras (cf. \cite{H}).  
In \cite{MM}, A. Mewborn and E. McMahon generalized the notion of  Hirata separable extensions as follows : 
$A/B$ is called a {\it strongly separable extension} if $V$ is finitely generated projective as $C$-module 
and an $A$-$A$-homomorphism from $A\otimes_BA$ to ${\rm Hom}^{\ell}(V_C, A_C)$ 
defined by $x \otimes y \mapsto [v \mapsto xvy]$ splits. 
It is obvious that a Hirata separable extension is strongly separable, and by \cite[Proposition 3.2]{MM}, 
a strongly separable extension is separable. 
In \cite{S3}, K. Sugano characterized strongly separable extensions as follows : 

\begin{Proposition}\label{SS}{\rm \cite[Theorem 1.1]{S3}}
A ring extension $A/B$ is strongly separable if and only if 
there exist $v_i \in V_A(B)$ and $\sum_j x_{ij} \otimes y_{ij} \in (A \otimes_B A)^A$ such that 
$u= \sum_{i,j}v_ix_{ij} u y_{ij}$ for all $u \in V_A(B)$.
\end{Proposition}

Let $S$ be an $A$-$A$-module, and $D:A \longrightarrow S$ an additive map. 
Then  $D$  is called a {\it $B$-derivation} of $A$ to $S$ if $D(xy)=D(x)y+xD(y)$  and $D(b)=0$ for all $b \in B$,
 $D$ is called {\it inner} if $D(x)=sx-xs$ for some fixed element $s \in S$, 
and $D$ is called {\it central} if $D(A) \subseteq S^A$.  
By \cite[Satz 4.2]{E}, $A/B$ is separable if and only if  for any $A$-$A$-module $S$, 
every $B$-derivation of $A$ to $S$ is inner.    
In \cite{N}, Y. Nakai introduced the notion of a quasi-separable extension of commutative rings  by using the module differentials, 
and in the noncommutative case, it was characterized by H. Komatsu \cite[Lemma 2.1]{Ko} as follows : 
$A/B$ is a {\it quasi-separable extension} if and only if for any $A$-$A$-module $S$, every central $B$-derivation of $A$ to $S$ is zero. 
Recently, in \cite{HN}, N. Hamaguchi and A. Nakajima generalized 
the notions of separable extensions and quasi-separable extensions as follows : 
$A/B$ is called a {\it weakly separable extension} if every $B$-derivation of $A$ to $A$ is inner,   
and $A/B$  is called a {\it weakly quasi-separable extension}   if every central $B$-derivation  of $A$ to $A$ is zero. 
Obviously, a separable extension is weakly separable and a quasi-separable extension is weakly quasi-separable. 
Moreover, a separable extension is quasi-separable by \cite[Theorem 2.1]{Ko}.

In \cite{M2}, Y. Miyashita introduced the notion of Morita invariance in ring extensions. 
Let $A/B$ and $A'/B'$ be ring extensions. 
Then we write $A/B \sim {A'}/{B'}$ if there exist Morita modules ${}_AM_{A'}$ and ${}_BN_{B'}$ 
such that ${}_AA \otimes_B N _{B'} \cong {}_AM_{B'}$.  
As was shown in \cite[Proposition 3.2]{M2}, $\sim$ is an equivalence relation. 
When this is the case,  we say that $A/B$ and $A'/B'$ are {\it Morita equivalent}. 
A class $\mathscr{C}$ of ring extensions is called {\it Morita invariant} if $\mathscr{C}$ has the following property : 
if $A/B$ is in $\mathscr{C}$ and $A/B \sim {A'}/{B'}$, then  ${A'}/{B'}$ is also in $\mathscr{C}$. 

In section 2, we shall recall some results which have been obtained in \cite{I0}, 
and we shall show some new results in the case of $A/B \sim {A'}/{B'}$. 
In section 3, we shall prove the classes of trivial extensions, liberal extensions, depth two extensions, strongly separable extensions, 
and weakly separable extensions  are Morita invariant. 
Finally, we will give an example of the class of ring extensions which is not Morita invariant.

\section{Preliminary results}
In this section, we assume that $A/B \sim {A'}/{B'}$, namely, there exist Morita modules ${}_AM_{A'}$ and ${}_BN_{B'}$ 
such that ${}_AA \otimes_B N_{B'} \cong {}_AM_{B'}$. 
We set $N^{*}={\rm Hom}^r({}_BN,{}_BB)$, and by $u^\rho$ we denote  the image of $u \in N$ by $\rho \in N^{*}$. 
Since ${}_BN \sim {}_BB$, there are two systems $\{g_k, n_k\}$ and $\{f_j , m_j\}$ ($g_k, \ f_j \in N^{*}$, $n_k, \ m_j \in N$) 
such that  $\sum_k {n_k}^{g_k}=1$ and $\sum_j n^{f_j} \cdot m_j =n$ for all $n \in N$. 
Moreover, we see that the map $\eta :N^{*} \otimes_B N \longrightarrow B'$ defined by $\rho \otimes u \mapsto \rho \cdot u_r$ 
($u_r$ is a right multiplication of $u$) is a $B'$-$B'$-isomorphism ($\eta^{-1}$ is given by $h \mapsto \sum_j f_j \otimes {m_j}^h$),
and the map $\xi : N \otimes_{B'} N^{*}  \longrightarrow B$ defined by $u \otimes \rho \mapsto u^\rho$ 
is a   $B$-$B$-isomorphism ($\xi^{-1}$ is given by $b \mapsto \sum_k b \cdot n_k \otimes g_k$).  
Now we can  define the map $\alpha : N^{*} \otimes_B A\otimes_B N \longrightarrow A'={\rm End}^r({}_AA\otimes_BN)$ 
 by $\rho \otimes x \otimes u \mapsto [y \otimes v \mapsto y \cdot v^\rho \cdot x \otimes u]$. 
Then $\alpha$ is a $B'$-$B'$-isomorphism, which induces a $B'$-$B'$-isomorphism $N^{*} \otimes_B B \otimes_B N \longrightarrow B'$. 
In fact,  $\alpha^{-1}$ is given by $h \mapsto \sum_j f_j \otimes (1 \otimes m_j)^h$. 
Moreover, we can define the multiplication among the elements of  $N^{*} \otimes_B A\otimes_B N$ by  
$(\rho \otimes x \otimes u) (\sigma \otimes y \otimes v) = \rho \otimes x \cdot u^\sigma \cdot y \otimes v$.  
Obviously, $\sum_j f_j \otimes 1 \otimes m_j$ is the identity element of the ring $N^{*} \otimes_B A \otimes_B N$, 
and $\alpha$ is a ring isomorphism.  
Thus, we can identify $A'$ and $B'$ with $N^{*} \otimes_B A \otimes_B N$ and $N^{*} \otimes_B B \otimes_B N$, respectively. 

Now we shall state some lemmas which have been obtained in \cite{I0}.  
The following lemma can be proved by a direct computation.  
\begin{Lemma}\label{L0}{\rm \cite[Lemma 1]{I0}}
${}_BA |{}_BB$ $($resp. ${}_BB|{}_BA$ $)$ implies  ${}_{B'}A' |{}_{B'}B'$ $($ resp. ${}_{B'}B'|{}_{B'}A'$ $)$.
\end{Lemma}

For any $x \in A$, we set ${x'}_{(j,k)} = f_j \otimes x \otimes n_k$ and ${x'}_{[j,k]}= g_k \otimes x \otimes m_j$. 
Then S. Ikehata considered the following maps : 
\begin{align*}
&\varphi : A \longrightarrow A', \ \ \varphi(x) = \sum_j f_j \otimes x \otimes m_j\\
&\phi : {\rm End}^{\ell}({}_BA_B) \longrightarrow {\rm End}^{\ell}({}_{B'}{A'}_{B'}), \ \ \phi (\eta ) = 1 \otimes \eta \otimes 1\\
&\psi : A \otimes_B A \longrightarrow A' \otimes_{B'} A', \ \ \psi (x \otimes y) = \sum_{j,k} {x'}_{(j,k)} \otimes {y'}_{[j,k]}.
\end{align*}

He proved the following lemmas (Lemma \ref{L2}, Lemma \ref{L3}, and Lemma \ref{L4})
 by making use of the above maps. These  are useful in our subsequent study.

\begin{Lemma}\label{L1}{\rm \cite[Lemma 2]{I0}}
$\varphi$ induces ring isomorphisms $V_A(B) \cong V_{A'}(B')$, $V_A(A) \cong V_{A'}(A')$, and $V_B(B) \cong V_{B'}(B')$.
\end{Lemma}

\begin{Lemma}\label{L2}{\rm \cite[Lemma 3]{I0}} \par
{\rm (1)} $\phi$ is a ring isomorphism. \par
{\rm (2)} $\phi$ induces an additive group isomorphism 
${\rm Hom}^{\ell}({}_BA_B, {}_BB_B) \cong {\rm Hom}^{\ell}({}_{B'}{A'}_{B'}, {}_{B'}{B'}_{B'}).$

{\rm (3)} $\phi$ induces an additive group isomorphism ${\rm Aut}^{\ell}(A/B) \cong {\rm Aut}^{\ell}(A'/B')$. \par
{\rm (4)} If $H$ is a subgroup of ${\rm Aut}^{\ell}(A/B)$ with $A^H =B$, then ${A'}^{\phi(H)} =B'$.
\end{Lemma}

\begin{Lemma}\label{L3}{\rm \cite[Lemma 4]{I0}}
$\psi$ induces an additive group isomorphism $(A \otimes_B A)^A \cong (A' \otimes_{B'} A')^{A'}$.
\end{Lemma}

To show the main results (Theorem \ref{D2} and Theorem \ref{WS}), we need some lemmas. 
First, we shall show the following. 

\begin{Lemma}\label{L4}
$\psi$ induces an additive group isomorphism $(A \otimes_B A)^B \cong (A' \otimes_{B'} A')^{B'}$.
\end{Lemma}

\begin{proof} Noting that $N \otimes_{B'} N^{*} \cong B$, we consider a $B'$-$B'$-isomorphism 
$\theta : A' \otimes_{B'} A' \longrightarrow N^{*} \otimes_B A \otimes_B A \otimes_B N$ 
defined by $(\rho \otimes x \otimes u) \otimes (\sigma \otimes y \otimes v) \mapsto \rho \otimes x \cdot u^\sigma \otimes y \otimes v$ 
($\theta^{-1}$ is given by $\rho \otimes x \otimes y \otimes u \mapsto \sum_k 
(\rho \otimes x \otimes n_k) \otimes (g_k \otimes y \otimes u)$). 
Let $\sum_r x_r \otimes y_r$ be in $(A \otimes_BA)^B$. 
For any $\sum_{\ell} \rho_{\ell} \otimes b_{\ell} \otimes u_{\ell} \in B'$, 
we have
\begin{align*}
& \theta\left( \sum_{\ell} \rho_{\ell} \otimes b_{\ell} \otimes u_{\ell} \cdot \psi \left( \sum_r x_r \otimes y_r \right) \right)\\
=& \  \theta\left( \sum_{\ell} \rho_{\ell} \otimes b_{\ell} \otimes u_{\ell} \cdot \sum_{r,j,k} (f_j \otimes x_r \otimes n_k) \otimes (g_k \otimes y_r \otimes m_j) \right)\\
=& \ \theta\left( \sum_{r,\ell,j,k} (\rho_{\ell} \otimes b_{\ell} \cdot {u_{\ell}}^{f_j} \cdot x_r \otimes n_k) \otimes (g_k \otimes y_r \otimes m_j) \right)\\
=& \ \sum_{r,\ell,j} \rho_{\ell} \otimes b_{\ell} \cdot u_{\ell}^{f_j} \cdot x_r \otimes y_r \otimes m_j\\
=& \ \sum_{\ell, j} \rho_{\ell} \otimes \left( \sum_r x_r  \otimes y_r \right) \cdot  b_{\ell}  \otimes u_{\ell}^{f_j} \cdot m_j\\
=& \ \sum_{r,\ell} \rho_{\ell} \otimes  x_r  \otimes y_r \cdot b_{\ell}  \otimes u_{\ell},
\end{align*}
and 
\begin{align*}
& \theta\left( \psi \left( \sum_r x_r \otimes y_r \right) \cdot \sum_{\ell} \rho_{\ell} \otimes b_{\ell} \otimes u_{\ell} \right)\\
=& \ \theta\left( \sum_{r,j,k} (f_j \otimes x_r \otimes n_k) \otimes (g_k \otimes y_r \otimes m_j) 
\cdot \sum_{\ell} \rho_{\ell} \otimes b_{\ell} \otimes u_{\ell} \right)\\
=& \ \theta\left( \sum_{r,\ell,j,k} (f_j \otimes x_r \otimes n_k) \otimes (g_k \otimes y_r \cdot  {m_j}^{\rho_{\ell}} \cdot b_{\ell} \otimes u_{\ell}) \right)\\
=& \ \sum_{r,\ell,j} f_j \otimes x_r \otimes  y_r \cdot  {m_j}^{\rho_{\ell}} \cdot b_{\ell} \otimes u_{\ell}\\
=& \ \sum_{\ell,j} f_j \cdot  {m_j}^{\rho_{\ell}}  \otimes \left( \sum_r x_r \otimes  y_r \right) \cdot  b_{\ell} \otimes u_{\ell}\\
=& \ \sum_{r,\ell} \rho_{\ell} \otimes  x_r  \otimes y_r \cdot b_{\ell}  \otimes u_{\ell}.
\end{align*}
It follows that $\sum_{\ell} h_{\ell} \otimes b_{\ell} \otimes u_{\ell} \cdot \psi \left( \sum_r x_r \otimes y_r \right)=\psi \left( \sum_r x_r \otimes y_r \right) \cdot \sum_{\ell} h_{\ell} \otimes b_{\ell} \otimes u_{\ell}$.
Thus, $\psi$ induces an additive group homomorphism from $(A \otimes_BA)^B$ to $(A' \otimes_{B'}{A'})^{B'}$. 
Assume that $\psi \left( \sum_r x_r \otimes y_r \right)=0$. Since  
\begin{align*}
0 = \theta \left( \psi \left( \sum_r  x_r \otimes y_r \right) \right)
&= \theta \left( \sum_{r,j,k} (f_j \otimes x_r \otimes n_k) \otimes (g_k \otimes y_r \otimes m_j) \right)\\
&= \sum_{r,j} f_j \otimes x_r \otimes y_r \otimes m_j,
\end{align*}
we obtain
\begin{align*}
0 = \sum_{r,j,k} {n_k}^{f_j} \cdot x_r \otimes y_r \cdot {m_j}^{g_k}
= \sum_{j,k} {n_k}^{f_j} \cdot {m_j}^{g_k} \left( \sum_r  x_r \otimes y_r \right)
= \sum_r x_r \otimes y_r.
\end{align*}
Finally, we shall prove $\psi\left( (A\otimes_BA)^B \right)=(A'\otimes_{B'}A')^{B'}$. 
Let $X'=\sum_r\left( \sum_\lambda \rho_{r \lambda} \otimes x_{r \lambda} \otimes u_{r \lambda} \right) \otimes 
\left( \sum_\mu \sigma_{r \mu} \otimes y_{r \mu} \otimes v_{r \mu} \right)$ be in $(A'\otimes_{B'}A')^{B'}$. 
Obviously, $(g_s \otimes b \otimes n_\ell) X'= X' (g_s \otimes b \otimes n_\ell) $ for any $b \in B$. 
Since 
\begin{align*}
(g_s \otimes b \otimes n_\ell) X' 
&= \sum_{r, \lambda, \mu} (g_s \otimes b \cdot {n_\ell}^{\rho_{r \lambda}} \cdot x_{r \lambda} \otimes u_{r \lambda})\otimes 
(\sigma_{r \mu} \otimes y_{r \mu} \otimes v_{r \mu})
\end{align*}
and
\begin{align*}
X'(g_s \otimes b \otimes n_\ell)
= \sum_{r, \lambda, \mu} (\rho_{r \lambda} \otimes x_{r \lambda} \otimes u_{r \lambda}) \otimes
(\sigma_{r \mu} \otimes y_{r \mu} \cdot {v_{r \mu}}^{g_s} \cdot b \otimes n_\ell),
\end{align*}
we can see that
$$
b \left( \sum_{r, \lambda, \mu,\ell} {n_\ell}^{\rho_{r \lambda}} \cdot x_{r \lambda} \cdot {u_{r \lambda}}^{\sigma_{r \mu}}
 \otimes y_{r \mu} \cdot {v_{r \mu}}^{g_\ell} \right)
 =\left( \sum_{r, \lambda, \mu, s} {n_s}^{\rho_{r \lambda}} \cdot x_{r \lambda} \cdot {u_{r \lambda}}^{\sigma_{r \mu}}
 \otimes y_{r \mu} \cdot {v_{r \mu}}^{g_s} \right) b.
$$
Hence, $X=\sum_{r, \lambda, \mu, \ell} {n_\ell}^{\rho_{r \lambda}} \cdot x_{r \lambda} \cdot {u_{r \lambda}}^{\sigma_{r \mu}}
 \otimes y_{r \mu} \cdot {v_{r \mu}}^{g_\ell}$ is in $(A \otimes_B A)^B$. Moreover, 
\begin{align*}
\psi(X)
 &= \sum_{r, \lambda, \mu, \ell, j, k} (f_j \otimes {n_\ell}^{\rho_{r \lambda}} \cdot x_{r \lambda} \cdot {u_{r \lambda}}^{\sigma_{r \mu}} \otimes n_k)
\otimes (g_k \otimes y_{r \mu} \cdot {v_{r \mu}}^{g_\ell} \otimes m_j)\\
&= \sum_{r, \lambda, \mu, \ell, j, k} (f_j \otimes {n_\ell}^{\rho_{r \lambda}} \cdot x_{r \lambda} \otimes u_{r \lambda})
(\sigma_{r \mu} \otimes 1  \otimes n_k)
\otimes (g_k \otimes y_{r \mu} \cdot {v_{r \mu}}^{g_\ell} \otimes m_j)\\
&= \sum_{r, \lambda, \mu, \ell, j, k} (f_j \otimes {n_\ell}^{\rho_{r \lambda}} \cdot x_{r \lambda} \otimes u_{r \lambda})
\otimes (\sigma_{r \mu} \otimes 1  \otimes n_k)
 (g_k \otimes y_{r \mu} \cdot {v_{r \mu}}^{g_\ell} \otimes m_j)\\
&= \sum_{r, \lambda, \mu, \ell, j} (f_j \otimes {n_\ell}^{\rho_{r \lambda}} \cdot x_{r \lambda} \otimes u_{r \lambda})
\otimes (\sigma_{r \mu} \otimes y_{r \mu} \cdot {v_{r \mu}}^{g_\ell} \otimes m_j)\\
&= \sum_{r, \lambda, \mu, \ell, j} (f_j \otimes 1 \otimes n_\ell)(\rho_{r \lambda} \otimes x_{r \lambda} \otimes u_{r \lambda})
\otimes (\sigma_{r \mu} \otimes y_{r \mu} \otimes v_{r \mu})(g_\ell \otimes 1 \otimes m_j)\\
&= \sum_{\ell, j} (f_j \otimes 1 \otimes n_\ell)(g_\ell \otimes 1 \otimes m_j)X'\\
&= X'.
\end{align*}
This completes the proof.
\end{proof}

Let $S$ be an $A$-$A$-module, and $S'=N^{*} \otimes_B S \otimes_B N$. 
Then $S'$ is an $A'$-$A'$-module. 
We set
\begin{align*}
{\rm Der}_B (A, S) &= \{ D \, | \, D:A \longrightarrow S, \ D {\rm \ is \ a} \ B{\rm \mathchar`- derivation}\}, \ {\rm and} \\
{\rm Der}_{B'} (A', S') &= \{ D' \, | \, D':A' \longrightarrow S', \ D' {\rm \  is \ a} \ B'{\rm \mathchar`- derivation}\}.
\end{align*}

Then we shall prove the following. 

\begin{Lemma}\label{L5}
$\phi$ induces an additive group isomorphism 
${\rm Der}_B (A, S) \cong {\rm Der}_{B'} (A', S')$. 
\end{Lemma}

\begin{proof} 
Let $D$ be in ${\rm Der}_B (A, S)$. For any $x'=\sum_\lambda \rho_\lambda \otimes x_\lambda \otimes u_\lambda$ and 
$y'=\sum_\mu \sigma_\mu \otimes y_\mu \otimes v_\mu \in A'$, we obtain
\begin{align*}
&\phi(D)(x'y')\\
&=\phi(D)\left( \sum_{\lambda,\mu} \rho_\lambda \otimes x_\lambda \cdot {u_\lambda}^{\sigma_\mu} \cdot y_\mu \otimes v_\mu \right)\\
&=  \sum_{\lambda, \mu} \rho_\lambda \otimes D(x_\lambda \cdot {u_\lambda}^{\sigma_\mu} \cdot y_\mu) \otimes v_\mu\\
&=  \sum_{\lambda, \mu} \rho_\lambda \otimes 
\left( D(x_\lambda) \cdot {u_\lambda}^{\sigma_\mu} \cdot y_\mu + x_\lambda \cdot {u_\lambda}^{\sigma_\mu} \cdot D(y_\mu) \right) \otimes v_\mu\\
&=  \sum_{\lambda, \mu}( \rho_\lambda \otimes D(x_\lambda) \cdot {u_\lambda}^{\sigma_\mu} \cdot y_\mu \otimes v_\mu
+ \rho_\lambda \otimes x_\lambda \cdot {u_\lambda}^{\sigma_\mu} \cdot D(y_\mu) \otimes v_\mu)\\
&=  \sum_{\lambda, \mu} \left( \left( \rho_\lambda \otimes D(x_\lambda) \otimes {u_\lambda} \right)(\sigma_\mu \otimes y_\mu \otimes v_\mu)
+ (\rho_\lambda \otimes x_\lambda \otimes u_\lambda)\left( \sigma_\mu \otimes D(y_\mu) \otimes v_\mu \right) \right)\\
&= \phi(D)(x')y'+x'\phi(D)(y').
\end{align*}
Moreover, it is obvious that $\phi(D)(B')=0$. 
Thus, $\phi$ induces an additive group homomorphism from 
${\rm Der}_B (A, S)$ to ${\rm Der}_{B'} (A', S')$. 
For any $D' \in {\rm Der}_{B'} (A', S')$,  
 we can  define an additive homomorphism $\widehat{D'}: A \longrightarrow S$ as follows : 
$\widehat{D'}(x)= \sum_{\lambda, k, \ell} {n_k}^{\rho_{ k \ell \lambda}} \cdot x_{ k \ell \lambda} \cdot {u_{ k \ell \lambda}}^{g_\ell}$, where 
$D'(g_k \otimes x \otimes n_\ell) = \sum_\lambda \rho_{ k \ell \lambda}  \otimes x_{ k \ell \lambda} \otimes u_{k \ell \lambda}$.  
Then we consider the additive homomorphism  $\pi :{\rm Der}_{B'} (A', S') \longrightarrow {\rm Hom}(A, S)$ defined by $\pi(D')=\widehat{D'}$. 
For any $D' \in  {\rm Der}_{B'} (A', S')$ and $x$, $y \in A$, we have
\begin{align*}
&D'(g_k \otimes xy \otimes n_\ell) \\
&= \, D'\left( \sum_s (g_k \otimes x \otimes n_s)(g_s \otimes y \otimes n_\ell) \right)\\
&= \, \sum_s\left( D'(g_k \otimes x \otimes n_s)(g_s \otimes y \otimes n_\ell) + (g_k \otimes x \otimes n_s)D'(g_s \otimes y \otimes n_\ell)   \right)\\
&= \, \sum_{\lambda,\mu,s} \left( (\rho_{k s \lambda} \otimes x_{k s \lambda} \otimes u_{k s \lambda})(g_s \otimes y \otimes n_\ell)
+ (g_k \otimes x \otimes n_s)(  \rho_{ s \ell \mu} \otimes y_{ s \ell \mu} \otimes u_{ s \ell \mu}     ) \right)\\
&= \, \sum_{\lambda , \mu, s} ( \rho_{k s \lambda} \otimes x_{k s \lambda} \cdot {u_{k s \lambda}}^{g_s} \cdot y \otimes n_\ell
+ g_k \otimes x \cdot {n_s}^{\rho_{ s \ell \mu}} \cdot y_{ s \ell \mu} \otimes u_{ s \ell \mu}).
\end{align*}
It follows that 
\begin{align*}
\pi(D')(xy) 
&= \sum_{\lambda , \mu, s, k, \ell} ( {n_k}^{\rho_{k s \lambda}} \cdot x_{k s \lambda} \cdot {u_{k s \lambda}}^{g_s} \cdot y \cdot {n_\ell}^{g_\ell}
+ {n_k}^{g_k} \cdot x \cdot {n_s}^{\rho_{ s \ell \mu}} \cdot y_{ s \ell \mu} \cdot {u_{ s \ell \mu}}^{g_\ell})\\
&= \left( \sum_{\lambda , \mu, s, k}  {n_k}^{\rho_{k s \lambda}} \cdot x_{k s \lambda} \cdot {u_{k s \lambda}}^{g_s} \right) y
+ x\left( \sum_{\lambda , \mu, s,  \ell}  {n_s}^{\rho_{ s \ell \mu}} \cdot y_{ s \ell \mu} \cdot {u_{ s \ell \mu}}^{g_\ell} \right)\\
&=\pi(D')(x)y + x \pi(D')(y).
\end{align*}
Moreover, it is clear that $\pi(D')(B)=0$. 
Thus, $\pi$ is an additive homomorphism from ${\rm Der}_{B'} (A', S')$ to ${\rm Der}_B (A, S)$. 
For any $D \in {\rm Der}_B (A, S)$ and $x \in A$,  
we see that $\pi \left( \phi(D) \right)(x) = D(x)$, and hence $\pi\phi=1_{{\rm Der}_B (A, S)}$.
For any $D' \in {\rm Der}_{B'} (A', S')$, we obtain 
\begin{align*}
&\phi\left( \pi(D') \right)(\sigma \otimes x \otimes v)\\
&= \,  \sigma \otimes \pi(D')(x) \otimes v\\
&= \,  \sigma \otimes \left( \sum_{\lambda, k, \ell} {n_k}^{\rho_{ k \ell \lambda}}  \cdot x_{ k \ell \lambda} \cdot {u_{k \ell \lambda}}^{g_\ell} \right)\otimes v\\
&= \, \sum_{k, \ell} (\sigma \otimes 1 \otimes  n_k)\left( \sum_\lambda \rho_{ k \ell \lambda} \otimes x_{ k \ell \lambda} \otimes u_{k \ell \lambda} \right)
(g_\ell \otimes 1 \otimes v)\\
&= \, \sum_{k, \ell} (\sigma \otimes 1 \otimes  n_k)D'(g_k \otimes x \otimes n_\ell)
(g_\ell \otimes 1 \otimes v)\\
&= \, D'\left( \sum_{k, \ell} (\sigma \otimes 1 \otimes  n_k)(g_k \otimes x \otimes n_\ell)
(g_\ell \otimes 1 \otimes v) \right)\\
&= \, D'(\sigma \otimes x \otimes v).
\end{align*}
This implies $\phi \pi =1_{{\rm Der}_{B'} (A', S')}$.  
This completes the proof.
\end{proof}

\section{Main results}

In the case of $A/B \sim A'/B'$, the conventions and notations employed in the preceding section 
will be used in this section. First, we shall state the following.

\begin{Proposition}
The class of trivial extensions is Morita invariant. 
\end{Proposition}

\begin{proof}  Assume that $A/B \sim A'/B'$ and $A/B$ is a trivial extension, that is, 
there is a $B$-$B$-module $S$ such that $A=B \oplus S$
 and the multiplication in $A$ is given by $(b,s)(c,t)=(bc,bt+sc)$. 
Obviously, $A'= B' \oplus S'$. 
Let $(b',s')=\left( \sum_\lambda\rho_\lambda \otimes b_\lambda \otimes u_\lambda,
 \, \sum_\lambda \rho_\lambda \otimes s_\lambda \otimes u_\lambda \right)$
 and $(c',t')=\left( \sum_\mu \sigma_\mu \otimes c_\mu \otimes v_\mu, \, 
\sum_\mu \sigma_\mu \otimes t_\mu \otimes v_\mu \right)$ 
be arbitrary elements in $A'$. 
Since 
$b'c'=\sum_{\lambda,\mu} \rho_\lambda \otimes b_\lambda \cdot {u_\lambda}^{\sigma_\mu} \cdot c_\mu \otimes v_\mu$ and
$b't'+s'c'=\sum_{\lambda,\mu}\rho_\lambda \otimes (b_\lambda \cdot {u_\lambda}^{\sigma_\mu} \cdot t_\mu + s_\lambda \cdot 
{u_\lambda}^{\sigma_\mu} \cdot c_\mu) \otimes v_\mu$,  
we have 
\begin{align*}
&(b',s')(c',t')\\
&= \left( \sum_\lambda \rho_\lambda \otimes (b_\lambda, s_\lambda) \otimes u_\lambda \right)
\left( \sum_\mu \sigma_\mu \otimes (c_\mu, t_\mu) \otimes v_\mu \right)\\
&= \sum_{\lambda,\mu}  \rho_\lambda \otimes (b_\lambda, s_\lambda) 
 ({u_\lambda}^{\sigma_\mu} \cdot c_\mu, \ {u_\lambda}^{\sigma_\mu} \cdot t_\mu) \otimes v_\mu\\
&= \sum_{\lambda,\mu}  \rho_\lambda \otimes (b_\lambda \cdot {u_\lambda}^{\sigma_\mu} \cdot c_\mu, \ 
b_\lambda \cdot {u_\lambda}^{\sigma_\mu} \cdot t_\mu + s_\lambda \cdot {u_\lambda}^{\sigma_\mu} \cdot c_\mu) \otimes v_\mu\\
&= \sum_{\lambda,\mu}  (\rho_\lambda \otimes b_\lambda \cdot {u_\lambda}^{\sigma_\mu} \cdot c_\mu \otimes v_\mu, \ 
\rho_\lambda \otimes (b_\lambda \cdot {u_\lambda}^{\sigma_\mu} \cdot t_\mu + s_\lambda \cdot 
{u_\lambda}^{\sigma_\mu} \cdot c_\mu) \otimes v_\mu)\\
&= (b'c', b't'+s'c').
\end{align*}
This completes the proof.
\end{proof}

In \cite{RS}, J. C. Robson and L. W. Small studied liberal extensions. 
Concerning liberal extensions, we shall show the following. 

\begin{Proposition}
The class of liberal extensions is Morita invariant.
\end{Proposition}

\begin{proof}  Assume that $A/B \sim A'/B'$ and $A/B$ is a liberal extension, that is, 
there is a finite set of elements $\{v_1, \ v_2, \ \cdots, \ v_n\}$ of  $V_A(B)$ such that $A=\sum_{i=1}^n v_iB$. 
By Lemma \ref{L1}, $v_i' = \sum_j f_j \otimes v_i \otimes m_j$  is in $V_{A'}(B')$.  
Let $x'=\sum_\lambda \rho_\lambda \otimes x_\lambda \otimes u_\lambda$ be in $A'$, 
and $x_\lambda = \sum_{i=1}^n v_i b_{\lambda i}$ ($b_{\lambda i} \in B$). 
Then we have
\begin{align*}
x'&= \sum_\lambda \rho_\lambda \otimes \left( \sum_{i=1}^n v_i b_{\lambda i} \right) \otimes u_\lambda\\
&= \sum_{i=1}^n \sum_\lambda \left( \sum_j f_j \cdot {m_j}^{\rho_\lambda} \right) \otimes v_i b_{\lambda i}\otimes u_\lambda\\
&= \sum_{i=1}^n \sum_\lambda \sum_j f_j \otimes v_i \cdot {m_j}^{\rho_\lambda} \cdot b_{\lambda i}\otimes u_\lambda\\
&= \sum_{i=1}^n \sum_\lambda \sum_j (f_j \otimes v_i \otimes m_j)(\rho_\lambda \otimes b_{\lambda i}\otimes u_\lambda)\\
&= \sum_{i=1}^n v'_i  \left( \sum_\lambda \rho_\lambda \otimes b_{\lambda i} \otimes u_\lambda \right).
\end{align*}
It follows that $A'=\sum_i v'_i B'$. This completes the proof.
\end{proof}

In \cite{KS}, L. Kadison and K. Szlach{\'a}nyi studied left and right depth two extensions 
and they gave  Proposition \ref{DT}. 
Then  we shall prove the following.  

\begin{Theorem}\label{D2}
The classes of left depth two extensions  and right depth two extensions are Morita invariant, respectively. 
\end{Theorem}

\begin{proof}  Assume that $A/B \sim A'/B'$ and $A/B$ is a left depth two extension. 
By Proposition \ref{DT} ,  
there exist $t_i \in (A\otimes_BA)^B$ and $\beta_i \in {\rm End}^\ell({}_BA_B)$ 
such that $\sum_i t_i \beta_i(x)y=x \otimes y$  for all $x$, $y \in A$.
Let $t_i = \sum_r x_{ir} \otimes y_{ir}$. 
By Lemma \ref{L2} and Lemma \ref{L4}, 
$\phi(\beta_i) = 1 \otimes \beta_i \otimes 1 \in {\rm End}^\ell({}_{B'}A'_{B'})$ and 
$\psi(t_i) = \sum_{r,j,k} (f_j \otimes x_{ir} \otimes n_k) \otimes (g_k \otimes y_{ir} \otimes m_j)  \in (A' \otimes_{B'} A')^{B'}$. 
Let $\theta$ be as in the proof of Lemma \ref{L4}. 
For any $z'=\sum_\lambda \rho_\lambda \otimes z_\lambda  \otimes u_\lambda$ and 
$w'=\sum_\mu \sigma_\mu \otimes w_\mu \otimes v_\mu$  in $A'$, 
we have 
$\theta(z' \otimes w') 
 = \sum_{\lambda, \mu} \rho_\lambda \otimes z_\lambda \otimes {u_\lambda}^{\sigma_\mu} \cdot w_\mu \otimes v_\mu$, 
and  
\begin{align*}
&\theta\left( \sum_i \psi(t_i) \phi(\beta_i)(z') w' \right)\\
&=  \theta\left( \sum_{i, r, \lambda, \mu, j, k}(f_j \otimes x_{ir} \otimes n_k) \otimes (g_k \otimes y_{ir} \otimes m_j)
\left( \rho_\lambda \otimes \beta_i(z_\lambda)  \otimes u_\lambda \right)(\sigma_\mu \otimes w_\mu \otimes v_\mu) \right)\\
&=  \theta\left( \sum_{i, r, \lambda, \mu, j, k}(f_j \otimes x_{ir} \otimes n_k) \otimes (g_k \otimes y_{ir} \cdot {m_j}^{\rho_\lambda}
 \cdot \beta_i(z_\lambda)  \cdot {u_\lambda}^{\sigma_\mu} \cdot w_\mu \otimes v_\mu) \right)\\
&=  \sum_{i,  \lambda, \mu, j} f_j \otimes t_i \cdot {m_j}^{\rho_\lambda}
 \cdot \beta_i(z_\lambda)  \cdot {u_\lambda}^{\sigma_\mu} \cdot w_\mu \otimes v_\mu \\
&=  \sum_{ \lambda, \mu, j} f_j \cdot {m_j}^{\rho_\lambda}\otimes \left( \sum_i t_i  
 \cdot \beta_i(z_\lambda)  \cdot {u_\lambda}^{\sigma_\mu} \cdot w_\mu \right) \otimes v_\mu\\
&=   \sum_{\lambda, \mu} \rho_\lambda \otimes z_\lambda \otimes {u_\lambda}^{\sigma_\mu} \cdot w_\mu \otimes v_\mu.
\end{align*}
This implies $\sum_i \psi(t_i)  \phi(\beta_i)(z') w'= z' \otimes w'$. 
Hence $A'/B'$ is a left depth two extension by Proposition \ref{DT}. 

By the same argument as above, we can prove that the class of right depth two extensions is Morita invariant. 
This completes the proof.
\end{proof}

In \cite{MM}, A. Mewborn and E. McMahon introduced the notion of strongly separable extensions 
as a generalization of  Hirata separable extensions. 
In \cite{S3}, K. Sugano characterized strongly separable extensions as Proposition \ref{SS}.  
Then  we shall prove the following.

\begin{Theorem}
The class of strongly separable extensions is Morita invariant.
\end{Theorem}

\begin{proof}    Assume that $A/B \sim A'/B'$ and $A/B$ is a strongly separable extension.
By Proposition \ref{SS},  
there exist $v_i \in V_A(B)$ and $\sum_r x_{ir} \otimes y_{ir} \in (A \otimes_BA)^A$ such that 
$u=\sum_{i,r} v_i x_{ir} u y_{ir}$ for any $u \in V_A(B)$. 
Let $v'_i = \sum_\lambda f_\lambda \otimes v_i \otimes m_\lambda$ and 
$\sum_r x'_{ir} \otimes y'_{ir} = \sum_{r,j,k} (f_j \otimes x_{ir} \otimes n_k) \otimes(g_k \otimes y_{ir} \otimes m_j)$. 
Then $v'_i \in V_{A'}(B')$ and $\sum_r x'_{ir} \otimes y'_{ir} \in (A' \otimes_{B'}A')^{A'}$ by Lemma \ref{L1} and Lemma\ref{L3}. 
For any $u' = \sum_\ell f_\ell \otimes u \otimes m_\ell \in V_{A'}(B')$ with $u \in V_A(B)$, we see that 
\begin{align*}
&\sum_{i,r} v'_i x'_{ir} u' y'_{ir}\\
&= \sum_{i,r,\lambda,j,k,\ell} (f_\lambda \otimes v_i \otimes m_\lambda)(f_j \otimes x_{ir} \otimes n_k)
(f_\ell \otimes u \otimes m_\ell)(g_k \otimes y_{ir} \otimes m_j)\\
&= \sum_{i,r,\lambda,j,k,\ell} (f_\lambda \otimes v_i \cdot {m_\lambda}^{f_j} \cdot x_{ir} \otimes n_k)
(f_\ell \otimes u \cdot {m_\ell}^{g_k} \cdot y_{ir} \otimes m_j)\\
&= \sum_{i,r,\lambda,j,k,\ell} (f_\lambda \cdot {m_\lambda}^{f_j} \otimes v_i  \cdot x_{ir} \otimes n_k)
(f_\ell \cdot {m_\ell}^{g_k} \otimes u  \cdot y_{ir} \otimes m_j)\\
&= \sum_{i,r,j,k} (f_j \otimes v_i  \cdot x_{ir} \otimes n_k)
(g_k \otimes u  \cdot y_{ir} \otimes m_j)\\
&= \sum_{i,r,j} f_j \otimes v_i  \cdot x_{ir} \cdot u  \cdot y_{ir} \otimes m_j\\
&= \sum_j f_j \otimes u \otimes m_j\\
&= u'.
\end{align*}
Hence $ A'/B'$ is a strongly separable extension by Proposition \ref{SS}. 
This completes the proof.
\end{proof}

N. Hamaguchi and A. Nakajima introduced the notion of weakly separable extensions 
as a generalization of separable extensions in their recent paper \cite{HN}.  
Concerning weakly separable extensions, we shall show  the following.

\begin{Theorem}\label{WS}
The class of weakly separable extensions is Morita invariant.
\end{Theorem}

\begin{proof} Assume that $A/B \sim A'/B'$ and $A/B$ is a weakly separable extension. 
Let $D'$ be in ${\rm Der}_{B'}(A',A')$. 
By Lemma \ref{L5}, $D'=1 \otimes D \otimes 1$ for some $D \in {\rm Der}_B(A,A)$.  
By the assumption, there exists $v \in A$ such that $D(x) = vx-xv$ ($x\in A$). 
Obviously, $v$ is in $V_A(B)$. Now, we put here $v'=\sum_j f_j \otimes v \otimes m_j$. 
Then, for any $x'=\sum_r \rho_r \otimes x_r \otimes u_r \in A'$, we see that 
\begin{align*}
D'(x') &=  \sum_r \rho_r \otimes D(x_r) \otimes u_r\\
&=   \sum_r \rho_r \otimes (vx_r-x_rv) \otimes u_r\\
&=   \sum_r \left( \left( \sum_i f_i \cdot {m_i}^{\rho_r} \right) \otimes v x_r \otimes u_r - \rho_r \otimes x_rv \otimes \left( \sum_j {u_r}^{f_j} \cdot m_j \right) \right)\\
&=    \sum_{r,i,j} ( f_i  \otimes v  \cdot {m_i}^{\rho_r} \cdot x_r \otimes u_r 
-   \rho_r \otimes x_r \cdot {u_r}^{f_j} \cdot v \otimes  m_j)\\
&= \sum_{r,i,j} ( (f_i \otimes v  \otimes m_i)(\rho_r \otimes x_r \otimes u_r) 
- (\rho_r \otimes x_r \otimes u_r)(f_j \otimes v \otimes  m_j))\\
&=   v'x'  - x'v'.
\end{align*}
Thus, $D'$ is an inner $B'$-derivation. This completes the proof.
\end{proof}

Finally, we shall show an example of a class of ring extensions which is not Morita invariant. 

{\bf Example.} 
Let $k$ be a field of a prime characteristic $p$, $A=k[t]$, and $B=k[t^p]$. 
Then  we see that $\{x^p \, | \, x \in A\} \subseteq B$. 
We set here 
$A'=\left[\begin{array}{cc}
A & A\\
A & A
\end{array}\right]$, $B'=\left[\begin{array}{cc}
B & B\\
B & B
\end{array}\right]$, 
${}_AM={}_AA \oplus {}_AA$, and ${}_BN={}_BB \oplus {}_BB$. 
One easily sees that ${\rm End}^r({}_AM)=A'$, ${\rm End}^r({}_BN)=B'$, and ${}_AA\otimes_BN_{B'} \cong {}_AM_{B'}$.
It follows that $A/B \sim A'/B'$.  
However,
$\{ (x')^n \, | \, x' \in A'\} \nsubseteq B'$ for any positive integer $n$. 
In fact,  $\left[\begin{array}{cc}
1 & t\\
0 & 0
\end{array}\right]^n = \left[\begin{array}{cc}
1 & t\\
0 & 0
\end{array}\right] \notin B'$. 
Consider the class of ring extensions which satisfies the following property : 
there exists a positive integer $n$ such that $\{ x^n \, | \, x \in A \} \subseteq B$. 
Then the class is not Morita invariant. 

\bigskip
{\bf Remark.}  
It may seem that the classes of almost all ring extensions  are Morita invariant. 
But, the author does not know whether the classes of quasi-separable extensions and weakly quasi-separable extensions
are  Morita invariant or not. 
Moreover,  $A/B$ is called a {\it finite normalizing extension} if there is a finite set of elements $\{a_1, a_2,\cdots, a_n \}$ of $A$ 
 such that $A = \sum_{i=0}^n a_i B$ and $a_i B = B a_i$ ($0 \leqq i \leqq n$). 
Then the author also does not know whether the class of finite normalizing extensions is Morita invariant or not.   

\bigskip
{\bf ACKNOWLEDGEMENTS.}  
The author wishes to thank Professor S. Ikehata with whose guidance and encouragement this work was done. 
The author  also would like to thank the referee for his valuable suggestions.


\end{document}